\documentclass[11pt]{amsart}
\usepackage{amsmath, amssymb, mathrsfs}
\usepackage{mathpazo}
\usepackage{color}
\def\lbr{\left\{}
\def\rbr{\right\}}

\def\msf{{\mathscr F}}

\def\me{{\mathbb  E}}

\def\mr{{\mathbb  R}}

\def\mn{{\mathbb  N}}

\def\mp{{\mathbb  P}}

\newcommand{\rf}[1]{(\ref{#1})}
\def\lnfrac#1#2{\raise.7ex \hbox{\Small $#1$}
  \kern-.15em/\kern-.15em  \lower.2ex \hbox{\Small $#2$}}

\theoremstyle{plain}
\newtheorem{theorem}{Theorem}[section]

\newtheorem{lem}[theorem]{Lemma}
\newtheorem{prop}[theorem]{Proposition}

\newtheorem{remark}[theorem]{Remark}

\numberwithin{equation}{section}

\usepackage{amsmath,amssymb,amscd,amsthm}

\usepackage{fullpage}

\usepackage[all,cmtip]{xy}
\usepackage[all]{xypic}

\usepackage{cite}

\usepackage{graphicx}
\usepackage{mathrsfs}

\newcommand{\smp}[1]{\left(#1 \right)}
\newcommand{\midp}[1]{\left[#1 \right]}
\newcommand{\bigp}[1]{\left\{#1 \right\}}

\newcommand{\R}{\mathbb{R}}

\newcommand{\N}{\mathbb{N}}

\newcommand{\E}{\mathcal{E}}

\newcommand{\divg}[1]{\mathrm{div}{\; #1}}

\providecommand{\abs}[1]{\left| #1\right|}
\providecommand{\norm}[1]{\left\lVert#1\right\rVert}

\usepackage{esint}

\newcommand{\ip}[1]{\langle #1 \rangle}

\title{Stochastic De Giorgi Iteration and Regularity of stochastic partial differential equations}

\author{Elton P. Hsu, Yu Wang and Zhenan Wang}
\address{}
\email{}

\begin{document}


\begin{abstract}
Under general conditions we show that the solution of a seimilinear stochastic parabolic partial differential equation of the form
\[
\partial_t u = \divg ( A \nabla u) + f (t,x, u) + g_i (t,x, u) \dot{w}^i_t
\]
with progressively measurable diffusion coefficients is almost surely H\"older continuous in both space and time variables.
\end{abstract}

\maketitle

\section{Introduction}
Stochastic partial differential equations (SPDEs) arise in many pure and applied sciences.  Regularity of solutions is of central importance for theoretical development as well as for numerical simulation. For linear equations with constant diffusion coefficients, $W^{k,2}$-theory has been well developed (see Pardoux~\cite{Par07} and Rozovskii~\cite{Roz90}), and a more general $W^{2,p}$-theory has been established by Krylov~\cite{Kry96}. Such equations can also be studied from a semigroup point of view (Brz\`ezniak, van Neerven, Veraar and Weis~\cite{BNVW08} and Da Prato and Zabczyk~\cite{PZ92}). Results concerning nonlinear equations can be found in Debussche, De Moor and Hofmanova~\cite{DMH14} and Pardoux~\cite{Par75}. In particular, many examples of semilinear SPDEs with measurable coefficients can be found in the survey monograph edited by Carmona and Rozovskii~\cite{Survey}. Although an obviously important question in applications,  regularity of solutions of semilinear SPDEs with random diffusion coefficients does not seem to have been adequately addressed in the literature.

In this paper we consider the following type of semilinear SPDEs on $\R^n$:
\begin{equation}
\label{basiceqn}
\partial_t u = \divg ( A \nabla u) + f (t,x, u) + g_i (t,x, u) \dot{w}^i_t,
\end{equation}
where $\{w^i\}$ is a sequence of independent standard Brownian motions on a filtered probability space $(\Omega, \msf_*, \mp)$, the diffusion coefficients $A$ are $\msf_* =\lbr\msf_t\rbr$-progressiviely measurable, and  $g = \{g_i\}$ is an $\ell^2$-valued function such that for each fixed $x$ and a progressively measurable process $h$, the process $g(t,x,h_t)$ is also progressively measurable. We will show that almost surely a stochastically strong solution with $L^2$-initial data is H\"older continuous in both space and (strictly positive) time variables and its H\"older norm has finite  moments of all orders.

The basic assumptions on the SPDE \eqref{basiceqn} are as follows:

(1) uniform ellipticity: $A(t, x; \omega)$ is $\msf_*$-progressively measurable and uniformly elliptic, i.e., there is a positive constant $\lambda$ such that
$$\lambda I\le A(t, x; \omega )\le \lambda^{-1}I, \quad  \forall  (t,x,\omega) \in \R_+\times \R^n \times \Omega .$$

(2) linear growth: there exist a nonnegative function $K\in L^2(\mr^n) \cap L^{\infty} (\R^n)$
and a positive constant $ \Lambda $ such that
\begin{equation}\label{assumption}
\abs{f(t,x,u)}+ \abs{g (t,x , u)}_{\ell^2} \leq K (x)+ \Lambda \abs{u} ,\quad \forall (t,x, u ) \in \R^+ \times \R^n \times\mr.
\end{equation}
We emphasize that no further conditions concerning the continuity $A, f$ or $g$ are imposed.
A stochastic process $u = u(t,x;\omega)$ is said to be a (stochastically strong) solution of \rf{basiceqn} if
it is an almost surely continuous $L^2$ process belonging to the space $L^2(\Omega\times \R^+, \mathscr{P}, W^{1,2}(\mr^n))$ and satisfies the SPDE \rf{basiceqn} in the sense that
\begin{equation}\label{interpret}
\ip{u (t) , \varphi} = \ip{u (0) , \varphi}- \int_0^t \ip{ A\nabla u (s),  \nabla \varphi } \;ds  + \int_{0}^t \ip{ f(u (s) ), \varphi }\;ds +\int_0^t \ip{ g_i(u(s) )  , \varphi}\,dw_s^i
\end{equation}
for all $\varphi \in C_{c}^{\infty} (\R^n)$. Here, $\mathscr{P}$ is the completion of the progressively measurable $\sigma$-algebra on $\Omega\times\mr_+$ under the product measure $\mp(d\omega)\times dt$, and $\langle\cdot, \cdot\rangle$ denotes the standard inner product on $L^2(\mr^n)$. The main result of the current work is the following moments estimate.

\begin{theorem}
Let $u$ be a (stochastically strong) solution of the SPDE \rf{basiceqn} with (non-random) initial data $u (0) =u_0$. Then for every $p>0$
there is a constant $C = C( n, \lambda, \Lambda, T,p)$ such that
\begin{equation*}
\me  \int_0^{2T} \norm{u (t)}_2^p \;dt  +
 \me \norm{u}_{\infty, [T,2T] \times \R^{n} }^p \leq  C\smp{  \norm{u_0}_2 + \Vert K\Vert_2+ \Vert K\Vert_\infty  }^p.
\end{equation*}
\end{theorem}

Using this moment estimate and following a suggestion from Professor Nicolai Krylov and the approaches used in Debussche, De Moor and Hofmanova\cite{DMH14}, we will prove the regularity for the solution.

\begin{theorem}
Let $u$ be a solution of the SPDE \rf{basiceqn} with a (deterministic) initial condition $u (0) =u_0\in L^2(\mr^n)$. Then there exists a positive exponent $\alpha=\alpha (n, \lambda, \Lambda)$ such that for all $T>0$ the solution $u \in C^{\alpha} ([T,2T] \times \R^n)$ almost surely. Furthermore, for every $p>0$, there is a constant $C=C (n, \lambda, \Lambda, T, p)$ such that
\[
\me \norm{u}_{C^{\alpha} ([T,2T] \times \R^n) }^p \leq C\smp{\norm{u_0}_{L^2(\mr^n)} +\Vert K\Vert_{L^2(\mr^n)} + \Vert K\Vert_{L^\infty(\mr^n)}}^p.
\]
\end{theorem}

\begin{remark}
In general, \eqref{basiceqn} may not admit a stochastically strong solution if the coefficients are merely progressively measurable. However, under some general conditions a weak solution exists, from which one can construct a strong solution on another probability space; see Carmona and Rozovskii~\cite{Survey} or Viot~\cite{Viot76} for a detailed exposition.
\end{remark}

The novelty of our result is that we do not impose any assumptions on the smoothness of $A, f$ or $g$.
Indeed, if $A$ and $g$ have some continuity, for example Dini continuity, then the above result follows directly from
Krylov\cite{Kry96, Kry99}. The approach we adopted in this work is quite different from the usual ones in the study of SPDEs.
Largely motivated by the recent work of Glatt-Holtz, \v{S}ver\'ak and Vicol \cite{GSV13} and Krylov \cite{Kry10}, rather than relying on abstract or explicit estimates of the solution kernel, we analyze the energy of the solution by a combination of PDE techniques and stochastic analysis. Indeed, our work can be viewed as a stochastic version of De Giorgi-Nash-Moser theory. As such our flexible method is potentially applicable to other types of nonlinear SPDEs.

The paper is orgainzed as follows. In {\sc Section} 2 we present a stochastic modification of De Giorgi's iteration. In {\sc Section} 3 we prove the decay of the tail probability of the solution. The main theorem stated above is proved in the last section.
\medskip

{\sc Acknowledgment.} The authors are very grateful to Professor Nicolai Krylov for the method used in {\sc Section} 4 of passing
from the $L^{\infty}$-bound to H\"older regularity. The authors thank Professors James Norris and  \'Etienne Pardoux for the electronic communications during the preparation of this work. They likewise thank Professors Luis Silvestre and Benjamin Gess for very helpful comments.

\section{Stochastic De Giorgi iteration}

De Giorgi's iteration is a classical method for studying elliptic and parabolic equations with measurable coefficients. In this section we develop a stochastic extension of this method appropriate for the type of SPDEs under investigation. See Cafarelli and Vasseur~\cite{CV10, CV10-1} for an exposition of the classical theory without random perturbation.

Throughout the paper, an $L^p$-norm without specifying a domain is implicitly assumed to be taken on $\R^n$; thus $\norm{ K}_p = \norm{ K}_{L^p(\R^n)}$. For  a time interval $I \subset \R^+$, we define the norm
\[
\norm{g}_{p_1, p_2, I} := \norm{ u }_{L^{p_1} ( I, L^{p_2 } (\R^n))} = \smp{ \int_{I} \norm{g}_{p_2}^{p_1} \;dt }^{1/p_1}.
\]
The norm most relevant for this paper is $\Vert \cdot\Vert_{4,2, I}$.

Let $I_k = [(1-2^{-k})T, 2T]$, a sequence of time intervals shrinking from $[0,2T]$ to $[T,2T]$. For each $a \geq 1$, write $u_{k,a} = (u-a(1-2^{-k}))^+$ and let
\[
 U_{k,a}:= \norm{u_{k,a}}_{4,2, I_k}^2 = \sqrt{\int_{I_k} \norm{u_{k,a}}_2^4\;dt}
\]
be the energy of $u$ on $I_k\times \mr^n$ above level $a(1-2^{-k})$.

For simplicity we will denote $f(t,x,u)$ and $g(t,x, u)$ by $f(u)$ and $g(u)$, respectively.  We have the following iterative inequality.

\begin{prop}
\label{S itr prop} Assume that the function $K(x)$ in the linear growth condition \eqref{assumption} satisfies $\Vert K\Vert_\infty\le 1$. Then for $n\geq 3$, there exists a constant  $C = C(n, \lambda, \Lambda, T)$ such that
\begin{equation}\label{S itr prop ineq}
U_{k,a} \leq \frac{C^k }{a^{2/(n+1)}} \smp{U_{k-1,a} + X^*_{k-1,a}} U_{k-1,a}^{1/(n+1)},
\end{equation}
where
\begin{equation}
\label{def X}
X_{k-1,a}^* = \sup_{(1-2^{-k})T\leq s\leq t \leq 2T}\int_s^t \ip{g_i (u(\tau)) , u_{k,a}(\tau) }\;dw_\tau^i.
\end{equation}
\end{prop}
\begin{proof} H\"older's inequality with the conjugate exponents $(n+1)/n$ and $n+1$ gives
\begin{equation}
\Vert u_{k,a} (t)\Vert_2^2  \leq \norm{ u_{k,a} (t)}^{2}_{2(n+1)/n} \cdot \abs{ \lbr u_{k,a} ( t) > 0 \rbr}^{1/(n+1)}.
\label{holder}
\end{equation}
Using Chebyshev's inequality, we have
\[
\abs{\lbr u_{k,a} ( t) > 0\rbr }= \abs{\lbr u_{k-1,a} ( t) > 2^{-k}a \rbr}\leq
\left(\frac{2^k}a\right)^2\norm{u_{k-1,a} (t)}^2_2.
\]
Squaring \rf{holder} and integrating with respect to $t$ on $I_k$ we have
\[
U_{k,a}^2 \leq \left(\frac{2^k}a\right)^{4/(n+1)}\int_{I_k} \norm{u_{k,a} (t)}^4_{2(n+1)/n} \norm{u_{k-1,a}(t) }^{4/(n+1)}_{2} \,dt.
\]
Applying H\"older's inequality again with the same conjugate exponents, we obtain
\begin{equation}
\label{r2 eq 1}
\begin{split}
U_{k,a} & \leq \left(\frac{2^k}a\right)^{2/(n+1)} \smp{ \int_{I_k} \norm{u_{k,a}(t)}^{4(n+1)/n}_{2(n+1)/n} dt }^{n/2(n+1)}  \smp{ \int_{I_k} \norm{u_{k-1,a}(t)}_2^4\;dt  }^{1/2(n+1)}.
\end{split}
\end{equation}
The third factor on the right side can be estimated by $U_{k-1,a}^{1/(n+1)}$.
The second factor is exactly $\Vert u_{k,a}\Vert_{4(n+1)/n, 2(n+1)/n, I_k}^2$.  We claim that
\begin{equation}
\Vert u_{k,a}\Vert_{4(n+1)/n, 2(n+1)/n, I_k}^2\le\sup_{t \in  I_{k}}\norm{u_{k,a} (t)}_{2}^2  +  \int_{I_{k}} \norm{ u_{k,a} (t) }_{2n/(n-2)}^2 \; dt. \label{interpolation}
\end{equation}
To prove this inequality we use the $L^p_tL^{q}_x$ interpolation inequality
$$\Vert u\Vert_{r_1,r_2, I}\le \Vert u\Vert^\gamma_{p_1, p_2, I}\Vert u\Vert^{1-\gamma}_{q_1, q_2, I}$$
with
$$\frac1{r_1} = \frac{\gamma}{p_1} + \frac{1-\gamma}{q_1}, \quad\frac{1}{r_2} = \frac{\gamma}{p_2}+\frac{1-\gamma}{q_2}.$$
Using this inequality with the parameters
$$r_1 = \frac{4(n+1)}n, \quad r_2 = \frac{2(n+1)}n, \quad p_1 = \infty, \quad q_1 = p_2 = 2,  \quad q_2 = \frac{2n}{n-2}, \quad \gamma = \frac{n+2}{2(n+1)}$$
followed by the elementary inequality
$$ab\le \frac{a^p}p + \frac {b^q}q\le a^p + b^q$$
with $p = 2(n+1)/(n+2)$ and $q = 2(n+1)/n$ we obtain \rf{interpolation} immediately.

Applying the Sobolev inequality on $\mr^n$ to the second term on the right side of \rf{interpolation} and then substituting the result in \eqref{r2 eq 1}, we obtain
\begin{equation}
\label{itr p 1}
U_{k,a} \leq C\left(\frac{2^k}a\right)^{2/(n+1)} \midp{ \sup_{t \in  I_{k}}\norm{u_{k,a} (t)}_{2}^2  +  \int_{I_{k}} \norm{\nabla  u_{k,a} (t) }_{2}^2dt}  U_{k-1,a}^{1/(n+1)}.
\end{equation}

We now  come to the key step of the proof, namely using It\^o's formula to bound the terms involving the supremum over $I_k$ and the gradient of $u$. The function $h_r(u) = \abs{ (u -r)^+}^2 $ is piecewise smooth with continuous derivative and its second derivative has a single point of discontinuity (a jump) at $u = r$. The quadratic variation process of the martingale part of the process $u (t)$ is absolutely continuous with respect to the Lebesgue measure on $\mr_+$. Thus, formally applying It\^o's formula and the SPDE \rf{basiceqn} to the composition $h_{a_k} ( u(t) ) = \vert u_{k,a}(t)\vert^2$ we have
\begin{align}\label{ito}
d\Vert u_{k,a}(t)\Vert^2_{2}= &-2\ip{\nabla u_{k,a}(t), A \nabla  u_{k,a} (t) } dt+  2 \ip{g_i (u)  , u_{k,a} (t)}  dw^i_t\\
&+\left[\int_{\R^n} \lbr \abs{g (u(t))}^2  +2u_{k,a}(t) f (u(t))\rbr 1_{\{ u_{k,a} (t)>0 \}} dx\right]  dt.\nonumber
&
\end{align}
The validity of the above application of It\^o's formula can be fully justified, see {\sc Remark} \ref{justify} below.

We now apply the uniform ellipticity assumption to the first term on the right side of \rf{ito}.  For the third term,  we observe that  if $u_{k,a}> 0$, then the inequalities
$1\leq a \leq 2^{k} u_{k-1,a}$ and $0<u \leq u_{k-1,a} + a \leq (1+ 2^{k}) u_{k-1,a}$ hold. By the linear growth condition (2) on $f$ and $g$, the fact $u_{k,a}\leq u_{k-1,a}$ and the assumption $\norm{K}_\infty \leq 1$, this term is bounded by $C^k\Vert u_{k-1,a}\Vert_2^2\, dt$ for some $C$. Now, integrating \rf{ito} from $t_0$ to $t$ with $t_0 \in  I_{k-1}\setminus I_{k}$ and $t \in  I_{k}$ gives
\begin{equation*}
\Vert u_{k,a}(t)\Vert^2_{2} + \lambda \int_{t_0}^t\norm{\nabla u_{k,a}(s)}_2^2ds \leq  \Vert u_{k,a} (t_0)\Vert_{2}^2+C^k U_{k-1,a} + \int_{t_0}^t\ip{g_i (u(s))  , u_{k,a}(s)}dw^i_s.
\end{equation*}
Taking the supremum over $t \in I_{k} $, we have for some constant $C$ depending only on $n,\lambda$ and $\Lambda$,
\begin{equation}
\label{Itr l2 2}
\begin{split}
& \sup_{ t\in I_k}\Vert u_{k,a} (t)\Vert_{2}^2  + \int_{t_0}^{2} \norm{\nabla u_{k,a} (s) }_{2}^2 \; ds \leq   C  \Vert u_{k,a} (t_0)\Vert_{2}^2+C^k U_{k-1,a}
+ CX^*_{k-1,a}
\end{split}
\end{equation}
with $X^*_{k-1, a}$ as defined in \eqref{def X}.
Noting the fact that $u_{k,a}\leq u_{k-1,a}$, we can find a $t_0 \in I_{k-1} \setminus I_{k} $ by the mean value theorem such that
\begin{equation}
\label{mvt}
\norm{u_{k,a} (t_0)}_{2}^2 = \frac{1}{\abs{ I_{k-1} \setminus I_{k}}  }\int_{I_{k-1}\setminus I_{k}} \norm{u_{k,a} (t)}_{2}^2 \;dt \leq 2^{k}T^{-1} U_{k-1,a}.
\end{equation}
Combining \eqref{itr p 1}, \eqref{Itr l2 2} and \eqref{mvt},  we obtain the desired iterative inequality \eqref{S itr prop ineq}.
\end{proof}

\begin{remark} \label{n=2 rem}
In the cases $n=1$ or $2$, the proof in this section shows that for any $\mu\in (0, 1/3)$, there is a constant $C = C(n, \lambda, \Lambda, T, \mu)$  such that
$$U_{k,a} \leq \frac{C^k }{a^{2 \mu}} \smp{U_{k-1,a} + X^*_{k-1,a}} U_{k-1,a}^{\mu}.$$
This is sufficient for estimating the tail probability of $\norm{u}_{\infty}$ in the next section, for all we need is that the factor $U_{k-1,a}$ carries an exponent strictly greater than $0$.
\end{remark}

\begin{remark}\label{justify}
For the justification of the It\^o expansion in \rf{ito}, we use a sequence $\varphi_\epsilon$ of smooth approximations of the function $h_r(u) = \vert (u-r)^+\vert^2$. In the definition \rf{interpret} of a solution, we use an approximation of the identity $\zeta_\delta$ as the test function. The desired expansion is obtained by letting $\delta\rightarrow0$ and then $\epsilon\rightarrow 0$. The details of these passing to the limit are very similar to those in Krylov~\cite{Kry10}.
\end{remark}

\section{Estimate of the tail probability}

In the context of the stochastic De Giorgi iteration, controlling the size of $\Vert u^+\Vert_{\infty, [T, 2T]\times\mr^n}$ means estimating the decay of the tail probability $\mp\lbr \Vert u\Vert_{\infty, [T, 2T]\times\mr^n}\ge a\rbr$. In order to use the iterative inequality in {\sc Proposition} \ref{S itr prop} for this purpose we need to show that $X_{k-1, a}^*$ is comparable with $U_{k-1, a}$.  This is accomplished in  {\sc Lemma} \ref{A-X} below, whose proof depends on the following simple result from stochastic analysis (see Norris~\cite[page 123]{Nor86}).

\begin{lem} \label{simple} Suppose that $\lbr M_t\rbr$ is a continuous local martingale. Then we have
$$\mp\lbr \sup_{0\le s\le t\le S}(M_t-M_s)\ge a, \, \langle M\rangle_S\le b\rbr\le 2e^{-a^2/4b}.$$
\end{lem}
\begin{proof} According to the Dambis, Dubins-Schwarz theorem (see Revuz and Yor \cite[Chapter V, Section 1, Theorem 1.6]{RY99}), there is a Brownian motion $B$ such that $M_t-M_0 = B_{\langle M\rangle_t}$, hence the event in the statement implies the event
$\displaystyle\bigp{ \sup_{0\le t\le b} B_t\ge a/2}$ or $\displaystyle\bigp{ \inf_{0\le t\le b} B_t\le -a/2}$.  Since $\sup_{0\le t\le b} B_t$ has the same distribution as $\sqrt b\vert B_1\vert$ by the reflection principle, we obtain the inequality from the explicit density function of a standard Gaussian random variable.
\end{proof}

Consider the continuous martingale
\begin{equation*} \label{spx}
X_t :=\int_{0}^t \ip{g_i (u(s)) , u_{k+1,a}(s)} \;dw^i_s
\end{equation*}
and recall from \rf{def X} that $X_{k,a}^* = \sup_{(1-2^{-k-1})T\le s\le t\le 2T}(X_t-X_s)$.

\begin{lem}\label{A-X} Assume that $\Vert K\Vert_\infty\leq 1$.  There exists a constant $C = C(n, \lambda,\Lambda)$ such that for all positive $\alpha$ and $\beta$,
\[
\mp\lbr X_{k,a}^*\ge\alpha\beta, \;  U_{k,a} \leq\beta\rbr \leq {C} \, e^{ -\alpha^2/C^k}.
\]
\end{lem}
\begin{proof} Let $T_k = (1-2^{-k-1})T$ for simplicity. If we can show that there is a constant $C$ such that
\begin{equation}
\label{A-Q}
\langle X\rangle_{2T} - \langle X\rangle_{T_k} \leq C^k U_{k,a}^2,
\end{equation}
then
$$\lbr X^*_{k, a}\ge\alpha\beta, U_{k, a}\le\beta\rbr\subset\lbr \sup_{T_k\le s\le t\le 2T}(X_t-X_s)\ge\alpha\beta,
\langle X\rangle_{2T} - \langle X\rangle_{T_k}\le C^k\beta^2\rbr$$
and the desired estimate follows immediately from {\sc Lemma} \ref{simple}.  To prove \rf{A-Q}, we start with
$$\langle X\rangle_{2T} - \langle X\rangle_{T_k} = \sum_{i\in\mn}\int_{I_k}\langle g_i(u), u_{k+1}\rangle^2\, ds,$$
which follows from the definition of $X_t$.
We observe that  if $u_{k+1,a}> 0$, then the inequalities
$1\leq a \leq 2^{k+1} u_{k,a}$ and $0<u \leq u_{k,a} + a \leq (1+ 2^{k+1}) u_{k,a}$ hold. By Minkowski's inequality (integral form), the linear growth condition (2) on $f$ and $g$ and the fact $u_{k+1,a}\leq u_{k,a}$ we have
\begin{equation*}
\label{A-X p1}
\sum_{i\in \N} \smp{  \int_{\R^n } g_i (u) \; u_{k+1,a} \; d x}^2  \leq  \smp{
\int_{\R^n} |g (u)|  u_{k+1,a} \;d x}^2 \leq C^k \smp{\int_{\R^n}
u_{k,a}^2 \;d x }^2.
\end{equation*}
Integrating over the interval $I_k$ we obtain the desired inequality \rf{A-Q}.
\end{proof}

Armed with the iterative inequality \rf{S itr prop ineq} and the comparison result {\sc Lemma} \ref{A-X} we are in a position to control the size of $\Vert u^+\Vert_{\infty, [T,2T]\times\mr^n}$ by estimating its tail probability. Without loss of generality we will only work with the case $T=1$. It is important that the constant $M_0$ in the following proposition is independent of $a$.

\begin{prop}
\label{Itr Nor} Assume that $\Vert K\Vert_\infty\le 1$. There exists a constant $ M_0 = M_0(n, \lambda, \Lambda)$ such that for all  $a \geq 1$ and $ M > M_0$,
\[
\mp\lbr\norm{u^+}_{\infty, [1,2] \times \R^n} > a, \; M \norm{u^+}_{4,2,[0,2]} \leq a\rbr \leq e^{-M^{\delta}}.
\]
Here, $\delta=1/(n+1)$ when $n\geq 3$ and $\delta$ can be any value from $(0,1/3)$ when $n=1$ or $2$.
\end{prop}
\begin{proof} As in the classical theory, we start with the observation that $\bigp{\norm{u^+}_{\infty, [1,2] \times \R^n}>a} \subset G_a^c$,
where $G_a = \bigp{\lim_{k \rightarrow \infty} U_{k,a} = 0}$. Consider the events $\mathcal{E}_{k} = \{U_{k,a} \leq (a/M)^2\gamma^k \}$ for a constant $\gamma<1$ to be determined later. Since $\norm{u}_{4,2,[0,2]} = \sqrt{U_{0, a}}$, it suffices to prove
\[
\mp\lbr G_a^c\cap\mathcal{E}_0\rbr\le e^{-M^\delta}.
\]
It is clear that
\[
G_{a}^c \subset \bigcup_{k \geq 0} \mathcal{E}_{k}^c  \subset  \mathcal{E}_{0}^c \cup \midp{ \bigcup_{k \geq 1} \smp{ \mathcal{E}_{k}^c \cap \mathcal{E}_{k-1} } },
\]
which implies
\begin{equation}
\label{Itr Nor 1}
\mp\lbr G_{a}^c \cap \E_{0}\rbr \leq \sum_{k \geq 1} \mp \lbr\mathcal{E}_{k}^c \cap \mathcal{E}_{k-1} \rbr.
\end{equation}
We estimate the probability  $\mp\lbr\mathcal{E}_{k}^c \cap \mathcal{E}_{k-1}\rbr $. We take $\alpha = (2C)^{k/2}M^\delta$ with the $C$ from {\sc Lemma} \ref{A-X}, and apply the lemma with this $\alpha$ and $\beta =  a^2\gamma^{k-1}/M^2$. If $ X^*_{k-1,a} \leq \alpha\beta$ and $U_{k-1,a} \leq \beta $, then by the iterative inequality \rf{S itr prop ineq} in {\sc Proposition \ref{S itr prop}} we have (after canceling $a^{2\delta}$!)
$$U_{k,a} \le\frac{C_1^k}{a^{2\delta}}(\beta + \alpha\beta)\beta^{\delta} = \frac{(C_1\gamma^{\delta})^k(1+(2C)^{k/2}M^\delta)}{\gamma^{1+\delta}M^{2\delta}}\cdot \gamma\beta\le\gamma\beta.$$
The last inequality holds if we choose $\gamma$ sufficiently small such that $(C_1\gamma^{\delta})^k(1+(2C)^{k/2}M^\delta)\le M^\delta$ for all $k\ge 1$ and $M\ge 1$ and then $M$ sufficiently large such that $\gamma^{1+\delta}M^\delta\ge 1$.

Now the above inequality implies that $\mathcal{E}_{k}^c \cap \mathcal{E}_{k-1 }  \subset \{ X^*_{k-1,a} >\alpha\beta, U_{k-1,a} \leq \beta \} $. Its probability is estimated by  {\sc Lemma} \ref{A-X} and we have
$$\mp \lbr \mathcal{E}_{k}^c \cap \mathcal{E}_{k-1} \rbr  \leq C e^{-\alpha^2/C^k} = Ce^{-2^{k}M^{2\delta}}.$$
Using this in \eqref{Itr Nor 1} we obtain, again for sufficiently large $M$,
\[
\mp\lbr G_{a}^c \cap \mathcal{E}_{0}^c\rbr \leq C\sum_{k=1}^{\infty} e^{-2^{k} M^{2\delta}} \leq   e^{- M^{\delta} }.
\]
This completes the proof of {\sc Proposition} \ref{Itr Nor}.
\end{proof}

\section{Moments estimate and H\"older continuity}
In this section, we first prove our main result, namely the moments estimate of the solution of the SPDE \rf{basiceqn} subject to the conditions stated in {\sc Section} 1. Then we will prove the almost surely H\"older continuity of the solution. We restate the moments estimate here.
\begin{theorem}
\label{Estimate}
Let $u$ be a (stochastically strong) solution of the SPDE \rf{basiceqn} with (non-random) initial data $u (0) =u_0$. Then for every $p>0$
there is a constant $C = C( n, \lambda, \Lambda, T,p)$ such that
\begin{equation*}
\me  \int_0^{2T} \norm{u (t)}_2^p \;dt  +
 \me \norm{u}_{\infty, [T,2T] \times \R^{n} }^p \leq  C\smp{  \norm{u_0}_2 + \Vert K\Vert_2+ \Vert K\Vert_\infty  }^p.
\end{equation*}
\end{theorem}
\begin{proof} By scaling it suffices to consider the case $T=1$,  $\Vert K\Vert_2+ \Vert K\Vert_\infty \le 1$, and $\norm{u_0}_2 \leq 1$. We need to show that there exists a constant $C$ (depending on $p$ of course) such that
\begin{equation}\label{4.1}
\me \int_0^{2} \norm{u (t)}_2^p \;dt  \leq  C\quad\text{and}\quad
 \me \norm{u}_{\infty, [1,2] \times \R^{n} }^p  \leq C.
\end{equation}
As $\mp$ is a probability measure, we may assume $p \geq 4$. We start with the first inequality.
Let $\varphi (t) = \norm{ u(t)}^2_2 +1$.
By It\^o's formula,
\begin{equation}
\label{cor1 p 1}
d \varphi (t)  = \varphi(t)(F(t)\;dt +d G_t),
\end{equation}
where
\[
F(t) =  \frac{ - \ip{A\nabla u , \nabla u} + \ip{f (u), u} +\norm{g (u)}_2^2}{ \norm{u}_2^2 +1} \quad \text{and}\quad G_t  = \int_0^t \frac{ \ip{ g_i (u), u }}{\norm{u}^2_2+1} \,dw_s^i.
\]
The solution of SDE \eqref{cor1 p 1} is explicitly given by
\[
\varphi (t)  = \varphi (0) \exp\midp{\int_0^t F(s) \;ds +G_t  - \frac{1}{2}\ip{G}_t }.
\]
By the assumptions we have $\ip{G}_t \leq 2(\Lambda+1)^2$ for all $t \leq 2$, therefore Novikov's condition ensures that
$$\exp\left[pG_t- \frac{p^2}{2}\ip{G}_t\right]$$
is a martingale for any $p>0$ and $0\leq t\leq 2$. This plus the fact $F(t) \leq 4 (\Lambda+1)^2$ give,
\begin{equation*}
\me \varphi^p (t)  = \varphi(0)^p\me\left[\exp p\smp{ \int_0^t F(s) \;ds +G_t- \frac{1}{2}\ip{G}_t}\right] \leq C \varphi(0)^p.
\end{equation*}
This implies the first inequality in \eqref{4.1}. Next, we show the second inequality in \rf{4.1}. Let
\[
X = \norm{u}_{\infty,[1,2]\times \R^n}\quad\text{and}\quad Y = \smp{\int_0^{2} \norm{u}_2^4\;dt }^{1/4}.
\]
By considering $u$ and $-u$ we have from {\sc Proposition} \ref{Itr Nor} with $\delta$ defined there,
\begin{equation}\label{yv}
\mp \lbr X > a , Y \leq  \frac{a}{M}\rbr \leq 2\,  e^{-M^\delta}
\end{equation}
for all $a \geq 1$ and $M \geq M_0$, hence
$$\mp\lbr X> a, Y\le\sqrt a\rbr\le 2\,e^{-a^{\delta/2}}$$
for $a\geq M_0^2$, assuming that $M_0\geq 1$. By the first inequality in \eqref{4.1}, we have
\[
\me Y^{2p} \leq  2^{(p-2)/2} \me  \int_0^2 \norm{u(t)}_2^{2p} \;dt \leq C.
 \]
Hence,
\[
\begin{split}
\me \norm{u}_{\infty, [1,2] \times \R^n}^p  &  =    p\int_0^\infty \mp (X > a) a^{p-1} \;da \\
& \leq  M_0^{2p} + p\int_{M^2_0}^\infty\mp\smp{ Y> \sqrt{a} } a^{p-1} \;da + p\int_{M_0^2}^\infty\mp\lbr X> a, Y\le \sqrt a\rbr a^{p-1}\;da.
\end{split}
\]
The second term is bounded by $\me Y^{2p}$, and the third term is finite by \rf{yv}.  This proves the second inequality in \eqref{4.1}.
\end{proof}

We can now prove the almost sure H\"older continuity result, which we state again for easy reference.
\begin{theorem}
\label{Thm Holder1} Let $u$ be a solution of the SPDE
$$\partial_t u = \divg ( A \nabla u) + f (t,x, u) + g_i (t,x, u) \dot{w}^i_t$$
whose coefficients satisfy the conditions stated in {\sc Section 1}. Then there exists a positive exponent $\alpha=\alpha (n, \lambda, \Lambda)$ such that almost surely $u \in C^{\alpha} ([T,2T] \times \R^n)$ for all $T>0$. Furthermore, for every $p>0$, there is a constant $C=C (n, \lambda, \Lambda, T, p)$ such that
\[
\me \norm{u}_{C^{\alpha} ([T,2T] \times \R^n) }^p \leq C\smp{\norm{u_0}_{L^2(\mr^n)} +\Vert K\Vert_{L^2(\mr^n)} + \Vert K\Vert_{L^\infty(\mr^n)}}^p.
\]
\end{theorem}
\begin{proof} By scaling it suffices to assume $T=1$, $ \norm{u_0}_2 \leq 1$ and $\Vert K\Vert_\infty+\Vert K\Vert_2\le 1$.  Following a suggestion of Professor Nicolai Krylov and the approaches used in Debussche, De Moor and Hofmanova\cite{DMH14}, we consider the solution $v$ of an SPDE with the same stochastic perturbation but simpler diffusion coefficients:
\[
d_t v  =\Delta v\, dt + g_i(u ) dw^i_t, \quad   v(2^{-1}) =0 .
\]
The function $\phi = u- v$ satisfies
\begin{equation}
\label{Holder p 1}
\partial_t \phi  = \divg (A \nabla \phi ) + f(\phi + v)  + \divg(A \nabla v) -\Delta v , \quad \text{ on } [2^{-1},2] \times \R^n.
\end{equation}
From  the linear growth assumption (1) for $g$ and {\sc Proposition} \ref{Estimate}, we have
\[
\me  \int_{2^{-1}}^{2} \norm{g(u)}_{p }^p \;dt\le C.
\]
According to Krylov's $W^{2,p}$-theory (see Krylov~\cite{Kry96}) $v \in C^{\alpha_1} ([2^{-1},2] \times \R^n)$ for some exponent $\alpha_1$. Furthermore, we have the estimates
\begin{equation}
\label{P H 1}
\me \norm{v}_{C^{\alpha_1} ( [2^{-1}, 2] \times \R^n )}^p \leq \me \norm{ g (u) }_{L^{p} ([2^{-1},2] \times \R^n )} \leq C
\end{equation}
and
\begin{equation}
\label{P H 3}
\me \int_{1/2}^2 \norm{D^2 v}_{W^{-1,p}}^p \;dt \le C_p.
\end{equation}
Since \eqref{Holder p 1} does not have a stochastic perturbation, the usual regularity theory (see Lieberman~\cite[Ch. VI]{Lieb}) applies and we have $\phi \in C^{\alpha_2} ( [1, 2] \times \R^n )$ for some small exponent $\alpha_2 \in (0,1) $ and
\[
\begin{split}
\norm{\phi }_{C^{\alpha_2} ( [1,2]\times \R^n )} & \leq C \smp{ \norm{\phi  }_{\infty, [2^{-1},2] \times \R^n}   +  \norm{D^2 v}_{L^p ([1,2] , W^{-1,p} )} } \\
& \leq C \smp{ \norm{u  }_{\infty, [2^{-1},2] \times \R^n} +\norm{v  }_{\infty, [2^{-1},2] \times \R^n}    +  \norm{D^2 v}_{L^p ([2^{-1},2] ,
W^{-1,p} )} }.
\end{split}
\]
Using the estimates \eqref{P H 1} and \eqref{P H 3} we conclude that
$\me \norm{\phi}^p_{C^{\alpha_2 } ([1,2] \times \R^n  )} \leq C$.
From this inequality, \eqref{P H 1} and $u = \phi+v$, we obtain the desired inequality $\me \norm{u}^p_{C^{\alpha} [1,2] \times \R^n} \leq C$
with $\alpha  = \min \{\alpha_1, \alpha_2\}$.
\end{proof}

\end{document}